\documentclass[12pt]{amsart}
\usepackage[top=1in, bottom=1in, left=1in, right=1in]{geometry}

\usepackage{amssymb,amsmath,amsthm}
\usepackage{mathrsfs} 
\usepackage{bbm} 
\usepackage{enumitem}
\usepackage[hyphens]{url}

\usepackage{graphicx}
\usepackage{color}

\numberwithin{equation}{section}

\theoremstyle{theorem}
\newtheorem{thm}{Theorem}
\newtheorem{lem}[thm]{Lemma}

\newtheorem{cor}[thm]{Corollary}

\theoremstyle{remark}
\newtheorem{rem}{Remark}

\newcommand{\Z}{\mathbb{Z}}

\newcommand{\R}{\mathbb{R}}
\newcommand{\C}{\mathbb{C}}

\newcommand{\eq}[2]{ \begin{equation} \label{#1}\begin{split} #2 \end{split} \end{equation} }
\newcommand{\al}[1]{\begin{align} #1 \end{align} }
\newcommand{\als}[1]{\begin{align*} #1 \end{align*} }

\newcommand{\nn}{\nonumber \\}

\newcommand{\dee}{\mathrm{d}}

\newcommand{\abs}[1]{\left|#1\right|}

\newcommand{\fl}[1]{\left\lfloor#1\right\rfloor}

\renewcommand{\hat}{\widehat}
\renewcommand{\tilde}{\widetilde}
\newcommand{\eps}{\varepsilon}
\renewcommand{\phi}{\varphi}

\renewcommand{\Re}{{\rm Re}}

\definecolor{red}{rgb}{1,0,0}
\definecolor{blue}{rgb}{.2,.6,.75}

\begin{document}

\title{Beyond the LSD method for the partial sums of multiplicative functions}

\author[A. Granville]{Andrew Granville}
\address{AG: D\'epartement de math\'ematiques et de statistique\\
Universit\'e de Montr\'eal\\
CP 6128 succ. Centre-Ville\\
Montr\'eal, QC H3C 3J7\\
Canada; 
and Department of Mathematics \\
University College London \\
Gower Street \\
London WC1E 6BT \\
England.
}
\thanks{A.G.~was funded by the European Research Council grant agreement n$^{\text{o}}$ 670239, and by the Natural Sciences and Engineering Research Council of Canada (NSERC) under the Canada Research Chairs program.}
\email{{\tt andrew@dms.umontreal.ca}}

\author[D. Koukoulopoulos]{Dimitris Koukoulopoulos}
\address{DK: D\'epartement de math\'ematiques et de statistique\\
Universit\'e de Montr\'eal\\
CP 6128 succ. Centre-Ville\\
Montr\'eal, QC H3C 3J7\\
Canada}
\thanks{D.K.~was funded by an NSERC Discovery grant, and by the Fonds de Recherche du Qu\'ebec, Nature et Technologies, as part of the program {\it \'Etablissement de nouveaux chercheurs universitaires}.}
\email{{\tt koukoulo@dms.umontreal.ca}}

\thanks{Both authors would like to thank Kevin Ford for helpful conversations, and Peter Humphries for his expert help on zero-free regions for exotic $L$-functions. They would also like to thank the anonymous referee for a careful reading of the paper and many helpful comments.}


\date{\today}

\begin{abstract} 
The Landau-Selberg-Delange (LSD) method gives an asymptotic formula for the partial sums of a multiplicative function $f$ whose prime values are $\alpha$ on average. In the literature, the average is usually taken to be $\alpha$ with a very strong error term, leading to an asymptotic formula for the partial sums with a very strong error term. In practice, the average at the prime values may only be known with a fairly weak error term, and so we explore here how good an estimate this will imply for the partial sums of $f$, developing new techniques to do so.
\end{abstract}

\maketitle


\section{Introduction}

Let $f$ be a multiplicative function whose prime values are $\alpha$ on average, where $\alpha$ denotes a fixed complex number. The prototypical such function is $\tau_\alpha$, defined to be the multiplicative function with Dirichlet series $\zeta(s)^\alpha$. We then easily check that $\tau_\alpha(p)=\alpha$ for all primes $p$ and, more generally, $\tau_\alpha(p^\nu)  =\binom{\alpha+\nu-1}{\nu}=\alpha(\alpha+1)\cdots (\alpha+\nu-1)/\nu!$. 

In order to estimate the partial sums of $\tau_\alpha$, we use Perron's formula: for $x\notin\Z$, we have
\[
\sum_{n\le x}\tau_\alpha(n)  =\frac{1}{2\pi i} \int_{\Re(s)=1+1/\log x} \zeta(s)^\alpha \frac{x^s}{s} \dee s .
\]
However, if $\alpha\notin\Z$, then the function $\zeta(s)^\alpha$ has an essential singularity at $s=1$, so the usual method of shifting the contour of integration to the left and using Cauchy's residue theorem is not applicable. 

A very similar integral in the special case when $\alpha=1/2$ was encountered by Landau in his work on integers that are representable as the sum of two squares \cite{landau}, as well as on his work counting the number of integers all of whose prime factors lie in a given residue class \cite{landau2}. Landau discovered a way to circumvent this problem by deforming the contour of integration around the singularity at $s=1$, and then evaluating the resulting integral using Hankel's formula for the Gamma function. His technique was further developed by Selberg \cite{selberg} and then by Delange \cite{del1,del2}. In its modern form, it permits us to establish a precise asymptotic expansion for the partial sums of $\tau_\alpha$ and for more general multiplicative functions. These ideas collectively form what we call the {\it Landau-Selberg-Delange method} or, more simply, the {\it LSD method}.\ \footnote{The method is often called the Selberg-Delange method, or even Selberg's method, but a key idea appears in Landau's work   long before Selberg's and Delange's papers. We would like to thank Steve Lester for bringing this to our attention. Moreover, we would like to thank Kevin Ford for pointing out paper \cite{landau2}.}

Tenenbaum's book \cite{Ten} contains a detailed description of the LSD method along with a general theorem that evaluates the partial sums of multiplicative functions $f$ satisfying a certain set of axioms. Loosely, if $F(s)$ is the Dirichlet series of $f$ with the usual notation $s=\sigma+it$, then the axioms can be rephrased as: (a) $|f|$ does not grow too fast; (b) there are constants $\alpha\in\C$ and $c>0$ such that $F(s)(s-1)^\alpha$ is analytic for $\sigma>1-c/\log(2+|t|)$. If $\tilde{c}_0,\tilde{c}_1,\dots$ are the Taylor coefficients of the function $F(s)(s-1)^{\alpha}/s$ about 1, then Theorem II.5.2 in \cite[pp. 281]{Ten} implies that
\eq{SD-classical}{
\sum_{n\le x} f(n)  = x \sum_{j=0}^{J-1} \tilde{c}_j\frac{ (\log x)^{\alpha-j-1} }{\Gamma(\alpha-j)}+ O_{J,f}\left(x(\log x)^{\Re(\alpha)-J-1}\right)
}
for each fixed $J$.

Our goal in this paper is to prove an appropriate version of the above asymptotic formula under the weaker condition 
\eq{f-alpha}{
\sum_{n\le x} f(p)\log p = \alpha x +  O\left( \frac{x}{(\log x)^A} \right)  
	\quad( x\ge 2 )  
}
for some $\alpha\in\C$ and some $A>0$. In particular, this assumption does not guarantee that $F(s)(s-1)^\alpha$ has an analytic continuation to the left of the line $\Re(s)=1$. It does guarantee however that $F(s)(s-1)^\alpha$ can be extended to a function that is $J$ times continuously differentiable in the half-plane $\Re(s)\ge1$, where $J$ is the largest integer $<A$. We then say that $F(s)(s-1)^\alpha$ has a {\it $C^J$-continuation} to the half-plane $\Re(s)\ge1$, and we set
\eq{taylor-coeffs}{
c_j = \frac{1}{j!}\cdot  \frac{\dee^j}{\dee s^j}\bigg|_{s=1} (s-1)^\alpha F(s) 
\quad\text{and}\quad
\tilde{c}_j = \frac{1}{j!}\cdot  \frac{\dee^j}{\dee s^j}\bigg|_{s=1} \frac{(s-1)^\alpha F(s)}{s} 
}
for $j\leq J$, the first $J+1$ Taylor coefficients about 1 of the functions $(s-1)^\alpha F(s)$ and $(s-1)^\alpha F(s)/s$, respectively. Since $s=1+(s-1)$ and, as a consequence, $1/s=1-(s-1)+(s-1)^2+\ldots $ for $|s-1|<1$, these coefficients are linked by the relations 
\[
\tilde{c}_j=\sum_{a=0}^j (-1)^a c_{j-a} 
\quad\text{and}\quad c_j = \tilde{c}_j+\tilde{c}_{j-1} 
\quad(0\le j\le J)
\]
with the convention that $\tilde{c}_{-1}=0$. Since $\zeta(s)\sim 1/(s-1)$ and $f$ is multiplicative, we also have that
\[
c_0=\tilde{c}_0 = \prod_p \left(1+\frac{f(p)}{p}+\frac{f(p^2)}{p^2}+\cdots\right)\left(1-\frac{1}{p}\right)^{\alpha} .
\]


\begin{thm}\label{SD} Let $f$ be a multiplicative function satisfying \eqref{f-alpha} and such that $|f|\le\tau_k$ for some positive real number $k$. If $J$ is the largest integer $<A$, and the coefficients $c_j$ and $\tilde{c}_j$ are defined by \eqref{taylor-coeffs}, then
\al{
\sum_{n\le x} f(n) 
	&= \int_2^x  \sum_{j=0}^J c_j  \frac{ (\log y)^{\alpha-j-1} } {\Gamma(\alpha-j)}\dee y
		+  O(x (\log x)^{k-1-A}(\log\log x)^{{\bf1}_{A=J+1}}) \label{asymp} \\
	&= x \sum_{j=0}^J \tilde{c}_j  \frac{ (\log x)^{\alpha-j-1} } {\Gamma(\alpha-j)}
		+  O(x (\log x)^{k-1-A}(\log\log x)^{{\bf1}_{A=J+1}}) \label{asymp2}   \ .
}
The implied constants depend at most on   $k$, $A$, and the implicit constant in \eqref{f-alpha}. The dependence on $A$ comes from both its size, and its distance from the nearest integer.
\end{thm}

We will demonstrate Theorem \ref{SD} in three successive steps, each one improving upon the previous one, carried out in Sections \ref{perron}, \ref{sec0} and \ref{proof}, respectively. Section \ref{aux} contains some preliminary results.

In Section \ref{best-possible}, we will show that there are examples of such $f$ with a term of size  $\gg x (\log x)^{\Re(\alpha)-1-A}$ in their asymptotic expansion, for  arbitrary $\alpha\in\C\setminus\Z_{\le0}$ and arbitrary positive non-integer $A>|\alpha|-\Re(\alpha)$. We deduce in Corollary \ref{best possible2} that  the error term in \eqref{asymp2}  is therefore best possible when $\alpha=k$ is a positive real number, and $A$ is not an integer.

The condition $|f|\le\tau_k$ can be relaxed significantly, but at the cost of various technical complications. We discuss such an improvement in Section \ref{divisor-bound}.

%

\medskip

Theorem \ref{SD} is of interest to better appreciate what ingredients go in to proving LSD-type results, which fits well with the recent development of the ``pretentious'' approach to analytic number theory in which one does not assume  the analytic continuation of $F(s)$. In certain cases, conditions of the form \eqref{f-alpha} are the best we can hope for. This is the case when $F(s)=L(s)^{1/2}$, where $L(s)$ is an $L$-function for which we only know a zero-free region of the form $\{s=\sigma+it \,:\, \sigma>1-1/(|t|+2)^{1/A+o(1)} \}$. Examples in which this is the best result known can be found, for instance, in the paper of Gelbart and Lapid \cite{GL}, and in the appendix by Brumley in \cite{Lap}.

\medskip

Wirsing, in the series   \cite{W1,W2},  obtained estimates for the partial sums of $f$ under the weaker hypothesis $\sum_{p\le x} (f(p)-\alpha)=o(x/\log x)$ as $x\to\infty$, together with various technical conditions ensuring that the values of $f(p)/\alpha$ are restricted in an appropriate part of the complex plane (these conditions are automatically met if $f\ge0$, for example). Since Wirsing's hypothesis is  weaker than  \eqref{f-alpha}, his estimate on the partial sums of $f$ is weaker than Theorem \ref{SD}. The methods of Section \ref{sec0} and \ref{proof} bear some similarity with Wirsing's arguments.

\medskip


\section{Initial preparations}\label{aux}

Let $f$ be as in the statement of Theorem \ref{SD}. Note that $|\alpha|\le k$. All implied constants here and for the rest of the paper might depend  without further notice on $k$, $A$, and the implicit constant in \eqref{f-alpha}. The dependence on $A$ comes from both its size, and its distance from the nearest integer.

\medskip

The first thing we prove is our claim that $F(s)(s-1)^\alpha$ has a $C^J$-continuation to the half-plane $\Re(s)\ge1$. To see this, we introduce the function $\tau_f$ whose Dirichlet series is given by $\prod_p (1-1/p^s)^{-f(p)}$, so that $f(p^\nu)=\binom{f(p)+\nu-1}{\nu}$ for all primes $p$ and all $\nu\ge1$. We also write $f=\tau_f*R_f$ and note that $R_f$ is supported on square-full integers and satisfies the bound $|R_f| = |f*\tau_{-f}|\le \tau_{2k}$. If $F_1$ and $F_2$ denote the Dirichlet series of $\tau_f$ and $R_f$, respectively, then $F_2(s)$ is analytic for $\Re(s)>1/2$. Hence our claim that $F(s)(s-1)^\alpha$ has a $C^J$-continuation to the half-plane $\Re(s)\ge1$ is reduced to the same claim for the function $F_1(s)(s-1)^\alpha$. This readily follows by \eqref{f-alpha} and partial summation, since
\eq{logF}{
\log[F_1(s)(s-1)^\alpha] = \sum_{p,\, \nu\ge1} \frac{f(p)-\alpha}{\nu p^{\nu s}} 
	+ \alpha \log[\zeta(s)(s-1)] .
}
  
Next, we simplify the functions $f$ we will work with. Define the function $\Lambda_f$ by the convolution formula 
\[
f\log=f*\Lambda_f.
\]
We claim that we may assume that $f=\tau_f$. Indeed, for the function $\tau_f$ introduced above, we have that $\Lambda_{\tau_f}(p^\nu) = f(p)\log p$\,; in particular, $|\Lambda_{\tau_f}|\le k\Lambda$. Moreover, if we assume that Theorem \ref{SD} is true for $\tau_f$, then we may easily deduce it for $f$: since $R_f$ is supported on square-full integers and satisfies the bound $|R_f|\le \tau_{2k}$, we have
\[
\sum_{n\le x}f(n) 
	= \sum_{ab\le x}\tau_f(a) R_f(b)
	=\sum_{b\le (\log x)^C} R_f(b) \sum_{a\le x/b} \tau_f(a)
	+ O( x(\log x)^{k-1-A} )  
\]
for $C$ big enough. Now, if Theorem \ref{SD} is true for $\tau_f$, then it also follows for $f$, since
\als{
\sum_{b\le (\log x)^C} \frac{R_f(b)}{b} \cdot \frac{\log^{\alpha-j-1}(x/b)}{\Gamma(\alpha-j)} 
	&=  \sum_{\ell=0}^J \frac{(\log x)^{\alpha-j-\ell-1}}{\ell!\Gamma(\alpha-j-\ell)} 
		\sum_{b\le (\log x)^C} \frac{R_f(b)(-\log b)^\ell}{b} \\
	&\qquad+ O((\log x)^{k-1-A}) \\
	&=  \sum_{\ell=0}^J \frac{(\log x)^{\alpha-j-\ell-1}}{\Gamma(\alpha-j-\ell)} \cdot \frac{F_2^{(\ell)}(1)}{\ell!}
			+ O((\log x)^{k-1-A})  
}
if $C$ is large enough. From now on, we therefore assume, without loss of generality, that $f=\tau_f$ so that the values of $f$ at $f(p^k)$ is determined by its value at $f(p)$, and in particular $|\Lambda_f|\le k\Lambda$. 

\medskip

Consider, now, the functions $Q(s):=F(s)(s-1)^\alpha$ and $\tilde{Q}(s) = Q(s)/s$. As we saw above, they both have a $C^J$-continuation to the half-plane $\Re(s)\ge1$. In particular, if $c_j$ and $\tilde{c}_j$ are given by \eqref{taylor-coeffs}, then for each $\ell\le J$ we have
\[
Q(s) = \sum_{j=0}^{\ell-1} c_j (s-1)^j + \frac{(s-1)^\ell}{(\ell-1)!} \int_0^1 Q^{(\ell)}(1+(s-1)u)(1-u)^{\ell-1}\dee u  .
\]
and
\[
\tilde{Q}(s) = \sum_{j=0}^{\ell-1} \tilde{c}_j (s-1)^j + \frac{(s-1)^\ell}{(\ell-1)!} \int_0^1 \tilde{Q}^{(\ell)}(1+(s-1)u)(1-u)^{\ell-1}\dee u  .
\]
To this end, we introduce the notations
\[
G_\ell(s) =  \sum_{j=0}^{\ell-1} c_j (s-1)^{j-\alpha} 
\quad\text{and}\quad 
\tilde{G}_\ell(s) =  \sum_{j=0}^{\ell-1} \tilde{c}_j (s-1)^{j-\alpha} ,
\]
as well as the ``error terms''
\eq{E}{
E_\ell(s) = F(s)-G_\ell(s) 
	= \frac{(s-1)^{\ell-\alpha}}{(\ell-1)!} \int_0^1 Q^{(\ell)}(1+(s-1)u)(1-u)^{\ell-1}\dee u ,
}
and
\eq{E2}{
\tilde{E}_\ell(s) = \frac{F(s)}{s} - \tilde{G}_\ell(s) 
	= \frac{(s-1)^{\ell-\alpha}}{(\ell-1)!} \int_0^1 \tilde{Q}^{(\ell)}(1+(s-1)u)(1-u)^{\ell-1}\dee u .
}
We have the following lemma:

\begin{lem}\label{L-bound} Let $f$ be a multiplicative function such that $f=\tau_f$ and for which \eqref{f-alpha} holds. Let also $s=\sigma+it$ with $\sigma>1$.
\begin{enumerate}[nolistsep]
\item Let $\ell\le J$, $m\ge0$, and $|s-1|\le2$. Then
\[
E_\ell^{(m)}(s) ,\, \tilde{E}_\ell^{(m)}(s) \ll  |s-1|^{\ell-\Re(\alpha)} (\sigma-1)^{-m} .
\]
\item Let $|s-1|\le2$ and $|t|\le (\sigma-1)^{1-\frac{A}{J+1}}/(-\log(\sigma-1))$. Then
\[
E_{J+1}^{(m)}(s) ,\, \tilde{E}_{J+1}^{(m)}(s) \ll |s-1|^{J+1-\Re(\alpha)} (\sigma-1)^{-(m+J+1-A)} (-\log(\sigma-1))^{{\bf1}_{A=J+1}}.
\]
\item Let $\ell\le J/2$, $m\ge0$, and $|s-1|\le2$. Then
\[
E_\ell^{(m+\ell)}(s)   ,\, \tilde{E}_\ell^{(m+\ell)}(s) \ll |s-1|^{-\Re(\alpha)} (\sigma-1)^{-m} \le 4^k (\sigma-1)^{-m-k} .
\]
\item Let $|t|\ge1$, $\ell\le J$, and $m\ge0$. Then
\[
F^{(m+\ell)}(s) \ll |t|^{\ell/A} (\sigma-1)^{-m-k} .
\]
\end{enumerate}
All implied constants depend at most on $k$, $A$ and the implicit constant in \eqref{f-alpha}. The dependence on $A$ comes from both its size, and its distance from the nearest integer.
\end{lem}

\begin{proof} Note that the functions $E_\ell(z)$ and $\tilde{E}_\ell(z)$ are holomorphic in the half-plane $\Re(z)>1$. In particular, they satisfy Cauchy's residue theorem in this region.

\medskip

(a)  From \eqref{logF} and \eqref{f-alpha}, we readily see that $Q^{(\ell)}(s) \ll 1$ uniformly when $\Re(s)\ge1$ and $|s-1|\le2$. Using the remainder formula \eqref{E}, we thus find that $E_\ell(s)\ll |s-1|^{\ell-\Re(\alpha)}$ when $\ell\le J$, $\Re(s)\ge1$ and $|s-1|\le2$. Thus Cauchy's residue theorem implies that
\eq{reduce to m=0}{
E_\ell^{(m)}(s) = \frac{m!}{2\pi i}\int_{|w|=(\sigma-1)/2} \frac{E_\ell(s+w)}{w^{m+1}}\dee w 
	&\ll  |s-1|^{\ell-\Re(\alpha)} (\sigma-1)^{-m} 
}
for $|s-1|\le2$, since $|s-1|/2\le |s-1+w|\le 3|s-1|/2$ when $|w|=(\sigma-1)/2\le |s-1|/2$. The bound for $\tilde{E}^{(m)}_\ell(s)$ is obtained in a similar way.

\medskip

(b) As in part (a), we focus on the claimed bound on $E_{J+1}^{(m)}(s)$, with the corresponding bound for $\tilde{E}^{(m)}_{J+1}(s)$ following similarly. Moreover, by the first relation in \eqref{reduce to m=0} with $\ell=J+1$, it is clear that is suffices to show the required bound on  $E_{J+1}^{(m)}(s)$ when $m=0$. 

Estimating $E_{J+1}(s)$ is trickier than estimating $E_\ell(s)$ with $\ell\le J$, because we can longer use Taylor's expansion for $Q$, as we only know that $Q$ is $J$ times differentiable. Instead, we will show that there are coefficients $c_0',c_1',\dots,c_J'$ independent of $s$ such that
\eq{taylorJ+1}{
\log Q(s) = \sum_{j=0}^J c_j' (s-1)^j + O(|s-1|^{J+1}(\sigma-1)^{A-J-1}(-\log(\sigma-1))^{{\bf1}_{J=A-1}}) 
}
when $|s-1|\le 2$. Notice that for $s$ as in the hypotheses of part (b), the error term is $\ll1$, so that the claimed estimate for $E_{J+1}(s)$ readily follows when $m=0$ by exponentiating \eqref{taylorJ+1} and multiplying the resulting asymptotic formula by $(s-1)^{-\alpha}$.

By our assumption that $|\Lambda_f|\le k\Lambda$, we may write $Q(s) = Q_1(s)Q_2(s)$, where $\log Q_1(s)=\sum_{p>3} (f(p)-\alpha)/p^s$ and $Q_2(s)$ is analytic and non-vanishing for $\Re(s)>1/2$ with $|s-1|\le2$. Thus, it suffices to show that $\log Q_1(s)$ has an expansion of the form \eqref{taylorJ+1}. Set $R(x)=\sum_{3<p\le x} (f(p)-\alpha) \ll x/(\log x)^{A+1}$ and note that
\[
\log Q_1(s) = s \int_e^\infty \frac{R(x)}{x^{s+1}} \dee x
	= s \int_1^\infty \frac{R(e^w)}{e^w} \cdot \frac{\dee w}{e^{w(s-1)}} . 
\]
Using Taylor's theorem, we find that
\[
e^{-w(s-1)} = \sum_{j=0}^J \frac{(w(1-s))^j}{j!} + \frac{(w(1-s))^{J+1}}{J!} \int_0^1 e^{-uw(s-1)} (1-u)^J \dee u ,
\]
so that
\als{
\log Q_1(s) = \sum_{j=0}^J \frac{s(1-s)^j}{j!} \int_1^\infty \frac{R(e^w)w^j}{e^w} \dee w 
	+ \frac{s(1-s)^{J+1}}{J!} \int_0^1 (1-u)^J \int_1^\infty \frac{R(e^w)w^{J+1}}{e^{w+uw(s-1)}}\dee w\dee u .
}
The last term is
\als{
&\ll |s-1|^{J+1} \int_0^1 \int_1^\infty \frac{w^{J-A}}{e^{uw(\sigma-1)}} \dee w \dee u  \\
&\ll |s-1|^{J+1} \int_0^1 (u(\sigma-1))^{A-J-1} \left(\log \frac{1}{u(\sigma-1)}\right)^{{\bf1}_{A=J+1}} \dee u \\
&\ll |s-1|^{J+1} (\sigma-1)^{A-J-1} \left(\log \frac{1}{\sigma-1}\right)^{{\bf1}_{A=J+1}} 
}
as needed, since $0<A-J\le 1$. This completes the proof of part (b) by taking
\[
c_j' = \frac{(-1)^j}{j!} \int_1^\infty \frac{R(e^w)w^j}{e^w} \dee w 
	+\frac{{\bf1}_{j\ge1}(-1)^{j-1}}{(j-1)!} \int_1^\infty \frac{R(e^w)w^{j-1}}{e^w} \dee w .
\]

\medskip

(c) Since $2\ell\le J$, we have that $Q^{(\ell+j)}(w)\ll1$ when $j\le\ell$ and $w\in\{z\in\C:\Re(z)\ge1,\,|z-1|\le2\}$. Differentiating the formula in \eqref{E} $\ell$ times, we thus conclude that
\[
E_\ell^{(\ell)}(s) \ll \sum_{j=0}^\ell |s-1|^{-\Re(\alpha)+\ell-j}\int_0^1 |Q^{(\ell+j)}(1+(s-1)u)| u^j(1-u)^{\ell-1}\dee u
	\ll |s-1|^{-\Re(\alpha)} .
\]
Since $|\alpha|\le k$ and $|s-1|\le2$, we find that $|s-1|^{k-\Re(a)}\le 2^{2k}$, whence
\[
|s-1|^{-\Re(\alpha)} \le 4^k |s-1|^{-k} \le 4^k(\sigma-1)^{-k}.
\]
The bound on $E_\ell^{(\ell+m)}(s)$ then by the argument in \eqref{reduce to m=0} with $E_\ell(s+w)$ replaced by $E_\ell^{(\ell)}(s+w)$. We argue similarly for the bound on $\tilde{E}^{(\ell+m)}_\ell(s)$.

\medskip

(d) Let $|t|\ge1$, and $j\le J$, and fix for the moment some $N\ge1$. Summation by parts implies and \eqref{f-alpha} imply that
\als{
\left(\frac{F'}{F}\right)^{(j-1)}(s)
	&= \sum_p \frac{f(p)(-\log p)^j}{p^s}\\
	&= O((\log N)^j) +(-1)^j \int_N^\infty \frac{(\log y)^{j-1}}{y^s} \dee (\alpha y +O(y/(\log y)^A)) \\
	&= O((1+|t|/(\log N)^A)(\log N)^j) + (-1)^j \alpha \int_N^\infty \frac{(\log y)^{j-1}}{y^s} \dee y  .
}
Moreover, we have
\[
\int_N^\infty \frac{(\log y)^{j-1}}{y^s} \dee y
	= \int_1^\infty \frac{(\log y)^{j-1}}{y^s} \dee y +O((\log N)^j)
	= \frac{(j-1)!}{(s-1)^j} +O((\log N)^j)\ll (\log N)^j 
\]
for $|t|\ge1$. Taking $\log N=|t|^{1/A}$ yields the estimate
\[
\left(\frac{F'}{F}\right)^{(j-1)}(s) \ll |t|^{j/A} .
\]
Now, note that $F^{(\ell)}/F$ is a linear combination of terms of the form $(F'/F)^{(j_1)}\cdots (F'/F)^{(j_\ell)}$ with $j_1+\cdots+j_\ell=\ell$. This can be proven by induction on $\ell$ and by noticing that
\[
 \frac{F^{(\ell+1)}}{F}  = \left(\frac{F^{(\ell)}}{F}\right)' + \frac{F'}{F} \cdot \frac{F^{(\ell)}}{F} .
\]
We thus conclude that
\[
\frac{F^{(\ell)}}{F}(s) \ll  |t|^{\ell/A} 
\]
Additionally, since $|f|\le\tau_k$, we have that  $|F(s)|\le \zeta(\sigma)^k \ll 1/(\sigma-1)^k$, whence $F^{(\ell)}(s) \ll |t|^{\ell/A}(\sigma-1)^{-k}$. The claimed estimate on $F^{(\ell+m)}(s)$ then follows by the argument in \eqref{reduce to m=0} with $E_\ell(s+w)$ replaced by $F^{(\ell)}(s)$.
\end{proof}

Finally, in order to calculate the main term in Theorem \ref{SD}, we need Hankel's formula for $1/\Gamma(z)$:

\begin{lem}\label{hankel}
For $x\ge1$, $c>1$ and $\Re(z)>1$, we have
\[
\frac{1}{2\pi i}\int_{\Re(s)=c} \frac{x^{s-1}}{(s-1)^z} \dee s
	={\bf1}_{x>1}\cdot \frac{(\log x)^{z-1}}{\Gamma(z)} .
\]
\end{lem}

 \begin{proof} Let $f(x) = {\bf1}_{x>1}(\log x)^{z-1}/\Gamma(z)$ and note that its Mellin transform is
\[
F(s):= \int_0^\infty f(x) x^{s-1} \dee x =  (-s)^{-z}
\]
for $\Re(s)<0$. By Mellin inversion we then have that $f(x) = \frac{1}{2\pi i}\int_{\Re(s)=c} F(s) x^{-s} \dee s$ for $c<0$. Making the change of variables $s\to 1-s$ completes the proof.

\medskip

Alternatively, we may give a proof when $x>1$ that avoids the general Mellin inversion theorem. We note that it suffices to prove that
\eq{hankel-alternative}{
\frac{1}{2\pi i}\int_{\Re(s)=c} \frac{x^{s+1}}{s(s+1)(s-1)^z} \dee s
	= \frac{1}{\Gamma(z)} \int_1^x \int_1^u (\log y)^{z-1} \dee y\, \dee u\  ,
}
since the claimed formula will then follow by differentiating with respect to $x$ and then with respect to $u$, which can be justified by the absolute convergence of the integrals under consideration. 

Using the formula
\eq{gamma}{
\int_1^\infty \frac{(\log y)^{z-1}}{y^s} \dee y  = \frac{\Gamma(z)}{(s-1)^z},
}
valid for $\Re(s)>1$, we find that
\als{
\frac{1}{2\pi i}\int_{\Re(s)=c} \frac{x^{s+1}}{s(s+1)(s-1)^z} \dee s
	&= \frac{x}{2\pi i}\int_1^\infty  \frac{(\log y)^{z-1}}{\Gamma(z)} \int_{\Re(s)=c} \frac{(x/y)^s}{s(s+1)} \dee s\,\dee y \\
	&= \frac{1}{\Gamma(z)} \int_1^x (\log y)^{z-1} (x-y) \dee y .
}
Since $x-y=\int_y^x \dee u$, relation \eqref{hankel-alternative} follows.
\end{proof}

\section{Using Perron's formula}\label{perron}

In this section, we prove a weak version of Theorem \ref{SD} using Perron's formula:


\begin{thm}\label{SD-weak} Let $f$ be a multiplicative function satisfying \eqref{f-alpha} and such that $|f|\le\tau_k$ for some positive real number $k$. If $\ell$ is the largest integer $<A/2$, and the coefficients $c_j$ and $\tilde{c}_j$ are defined by \eqref{taylor-coeffs}, then
\al{
\sum_{n\le x} f(n) 
	&= \int_2^x  \sum_{j=0}^{\ell-1} c_j  \frac{ (\log y)^{\alpha-j-1} } {\Gamma(\alpha-j)}\dee y
		+  O(x (\log x)^{k-\ell})  \label{asymp-weak} \\
	&= x \sum_{j=0}^{\ell-1} \tilde{c}_j  \frac{ (\log x)^{\alpha-j-1} } {\Gamma(\alpha-j)}
		+  O(x (\log x)^{k-\ell}) \label{asymp-weak2}   \ .
}
The implied constants depend at most on  $k$, $A$, and the implicit constant in \eqref{f-alpha}. 
\end{thm}

\begin{proof} As we discussed in Section \ref{aux}, we may assume that $f=\tau_f$. We may also assume that $A>2$, so that $\ell\ge1$; otherwise, the theorem is trivially true. 

We fix $T\in[\sqrt{\log x},e^{\sqrt{\log x}}]$ to be chosen later as an appropriate power of $\log x$, and we let $\psi$ be a smooth function supported on $[0,1+1/T]$ with 
\[
 \begin{cases}
\psi(y)=1&\text{if}\ y\le 1,\\
\psi(y)\in[0,1] &\text{if}\ 1<y\le 1+1/T,\\
\psi(y)=0 &\text{if}\ y>1+1/T ,
\end{cases}
\]
and whose derivatives satisfy for each fixed $j$ the growth condition $\psi^{(j)}(y) \ll_j T^j$ uniformly for $y\ge0$. For its Mellin transform, we have the estimate
\eq{Psi-e1}{
\Psi(s) = \int_0^\infty \psi(y) y^{s-1} \dee y
	= \frac{1}{s} + \int_1^{1+1/T} \psi(y)y^{s-1}\dee y = \frac{1}{s}+O\left(\frac{1}{T}\right) \quad(1\le\sigma\le2) .
}
This estimate is useful for small values of $t$. We also show another estimate to treat larger values of $t$. Integrating by parts, we find that 
\[
\Psi(s) = - \frac{1}{s}\int_0^\infty \psi'(y)y^s\dee y 
	= - \frac{1}{s}\int_1^{1+1/T} \psi'(y)y^s\dee y \quad(1\le\sigma\le2) .
\]
Iterating and using the bound $\psi^{(j)}(y) \ll_j T^j$, we find that
\[
\Psi(s)  = \frac{(-1)^j}{s(s+1)\cdots(s+j-1)}\int_1^{1+1/T} \psi^{(j)}(y)y^{s+j-1}\dee y 
	\ll_j \frac{T^{j-1}}{|t|^j} \quad(1\le\sigma\le2)  .
\]
We thus conclude that
\eq{Psi-e2}{
\Psi(s) \ll_j \frac{1}{|t|\cdot(1+|t|/T)^{j-1}} \quad(1\le\sigma\le2,\ j\ge1) .
}

Now, let $r$ denote an auxiliary large integer. Then
\als{
\sum_{n\le x} f(n)(\log n)^{r+2\ell}
	&= \sum_{n=1}^\infty f(n)(\log n)^{r+2\ell}\psi(n/x)  
	+ O\left( \sum_{x<n\le x+x/T}|f(n)|(\log n)^{r+2\ell}\right)  \\
	&= \frac{(-1)^r}{2\pi i} \int_{\sigma=1+1/\log x} F^{(r+2\ell)}(s) \Psi(s) x^s \dee s 
		  +O\left(\frac{x(\log x)^{r+2\ell+k-1}}{T}\right) 
} 
since $|f(n)|\le \tau_k(n)$. Fix $\eps>0$. When $|t|\ge (\log x)^\eps T$, we use the bound $\Psi(1+1/\log x+it) =O(T^{j-1}/|t|^j)$ with $j\ge(r+2\ell+k)/\eps+1$. Since we also have that $F^{(r+2\ell)}(1+1/\log x+it)=O((\log x)^{k+r+2\ell})$, we find that 
\[
\sum_{n\le x} f(n)(\log n)^{r+2\ell} 
	= \frac{(-1)^{r}}{2\pi i} \int_{\substack{\sigma=1+1/\log x \\ |t|\le (\log x)^\eps T }} 
		F^{(r+2\ell)}(s) \Psi(s) x^s \dee s 
		    +O\left(x+\frac{x(\log x)^{r+2\ell+k-1}}{T}\right)  .
\]

For $s=1+1/\log x+it$ with $1\le|t|\le (\log x)^\eps T$, we use the bounds $\Psi(s)\ll 1/|t|$ and $F^{(r+2\ell)}(s)\ll |t|^{2\ell/A}(\log x)^{k+r}$, with the second one following from Lemma \ref{L-bound}(d) with $m=r$ and $2\ell$ in place of $\ell$. Thus
\als{
\sum_{n\le x} f(n)(\log n)^{r+2\ell} 
	&= \frac{(-1)^{r}}{2\pi i} \int_{\substack{\sigma=1+1/\log x \\ |t|\le1}}
		F^{(r+2\ell)}(s) \Psi(s) x^s \dee s  \\
	&\qquad +O\left( \frac{x(\log x)^{k+r+2\ell-1}}{T} + x(\log x)^{k+r}\cdot ((\log x)^\eps T)^{2\ell/A} \right) .
}
Since we have assumed that $T\ge \sqrt{\log x}$ and $2\ell<A$, we have that $((\log x)^\eps T)^{2\ell/A}\le T$ for $\eps$ small enough, so that
\als{
\sum_{n\le x} f(n)(\log n)^{r+2\ell} 
	&= \frac{(-1)^{r}}{2\pi i} \int_{\substack{\sigma=1+1/\log x \\ |t|\le1}}
		F^{(r+2\ell)}(s) \Psi(s) x^s \dee s  \\
	&\qquad +O\left( \frac{x(\log x)^{k+r+2\ell-1}}{T} + x(\log x)^{k+r}T \right) .
}

In the remaining part of the integral, we use the formula $\Psi(s)=1/s+O(1/T)$ and the bound $F^{(r+2\ell)}(s)\ll(\log x)^{r+2\ell+k}$ to find that
\als{
\sum_{n\le x} f(n)(\log n)^{r+2\ell} 
	&= \frac{(-1)^{r}}{2\pi i} \int_{\substack{\sigma=1+1/\log x \\ |t|\le1}}
		F^{(r+2\ell)}(s) \frac{x^s}{s} \dee s  \\
	&\qquad +O\left( \frac{x(\log x)^{k+r+2\ell}}{T} + x(\log x)^{k+r}T \right) .
}
We then choose $T=(\log x)^\ell$ and use Lemma \ref{L-bound}(c) with $m=r+\ell$ to write $F^{(r+2\ell)}(s) = G_\ell^{(r+2\ell)}(s)+O((\log x)^{r+\ell+k})$. Hence
\als{
\sum_{n\le x} f(n)(\log n)^{r+2\ell} 
	&= \frac{(-1)^{r}}{2\pi i} \int_{\substack{\sigma=1+1/\log x \\ |t|\le1}}
		G_\ell^{(r+2\ell)}(s) \frac{x^s}{s} \dee s \\
	&\qquad +O\left( \frac{x(\log x)^{k+r+2\ell}}{T} + x(\log x)^{k+r}T \right) \\
	&= \frac{(-1)^{r}}{2\pi i} \int_{\substack{\sigma=1+1/\log x \\ |t|\le1}}
		G_\ell^{(r+2\ell)}(s) \frac{x^s}{s} \dee s +O(x(\log x)^{r+k+\ell})  .
}
 
Note that $G_\ell^{(r+2\ell)}(s) \ll  |s-1|^{-\Re(\alpha)-2\ell-r}+|s-1|^{-\Re(\alpha)-\ell-r-1}$. Thus, if $r\ge|\alpha|+1$, then both exponents of $|s-1|$ are $\le-2$. In particular, $G_\ell^{(r+2\ell)}(s) \ll |t|^{-2}$ when $|t|\ge1$ and $G_\ell^{(r+2\ell)}\ll (\sigma-1)^{-2r-2\ell}$ otherwise, so that
\als{
\sum_{n\le x} f(n)(\log n)^{r+2\ell} 
  	&=  \frac{(-1)^{r}}{2\pi i} \int_{\sigma=1+1/\log x} 
		G_\ell^{(r+2\ell)}(s) \frac{x^s}{s} \dee s +O(x(\log x)^{r+k+\ell}) \\
	&=  \frac{(-1)^{r}}{2\pi i} \int_{\sigma=1+1/\log x} 
		G_\ell^{(r+2\ell)}(s) \frac{x^s-1}{s} \dee s +O(x(\log x)^{r+k+\ell}) .
}

Since $(x^s-1)/s = \int_1^x y^{s-1}\dee y$ and 
\[
(-1)^r G_\ell^{(r+2\ell)}(s) = \sum_{j=0}^{\ell-1} \frac{\Gamma(\alpha-j+r+2\ell)}{\Gamma(\alpha-j)} c_j (s-1)^{-\alpha-r-2\ell+j} , 
\]
we find that 
\als{
&\frac{(-1)^{r}}{2\pi i} \int_{\substack{\sigma=1+1/\log x }}
		G_\ell^{(r+2\ell)}(s) \frac{x^s-1}{s} \dee s\\
	&\qquad= \sum_{j=0}^{\ell-1} \frac{\Gamma(\alpha-j+r+2\ell)}{\Gamma(\alpha-j)} 
	\cdot \frac{c_j}{2\pi i} \int_1^x 	
		\int_{\substack{\sigma=1+1/\log x }} (s-1)^{-\alpha-r-2\ell+j} y^{s-1} \dee s\, \dee y  \\
	&\qquad=\sum_{j=0}^{\ell-1} \frac{c_j}{\Gamma(\alpha-j)}  \int_1^x  (\log y)^{\alpha+r+2\ell-j-1}  \dee y 
}
by Lemma \ref{hankel}, whence
\als{
\sum_{n\le x} f(n)(\log n)^{r+2\ell}
	&= \int_1^x  \sum_{j=0}^{\ell-1} \frac{c_j}{\Gamma(\alpha-j)} (\log y)^{\alpha+r+2\ell-j-1}  \dee y 
		+O(x(\log x)^{r+\ell+k}) .
}
Partial summation the completes the proof of \eqref{asymp-weak}.

To deduce \eqref{asymp-weak2}, we integrate by parts in \eqref{asymp-weak}. Alternatively, we may use a modification of the argument leading to \eqref{asymp-weak}, starting with the formula
\als{
\sum_{n\le x} f(n)\psi(n/x)
	&= \frac{1}{2\pi i} \int_{\sigma=1+1/\log x} F(s) \Psi(s) x^s \dee s \\
	&= \frac{(-1/\log x)^{r+2\ell}}{2\pi i} \int_{\sigma=1+1/\log x} (F\Psi)^{(r+2\ell)}(s) x^s \dee s ,
}
that is obtained by integrating by parts $r+2\ell$ times. We then bound the above integral as before: in the portion with $|t|\ge1$, we estimate $F$ and its derivatives by Lemma \ref{L-bound}(d), and we use the bound $\Psi^{(j)}(s) \ll_j |t|^{-1}/(1+|t|/T)^{j-1}$; in the portion with $|t|\le1$, we use the bound $\frac{\dee^j}{\dee s^j}(\Psi(s)-1/s) \ll 1/T^{j+1}$ and we approximate $(F(s)/s)^{(r+2\ell)}$ by $\tilde{G}_\ell^{(r+2\ell)}(s)$ using Lemma \ref{L-bound}(c). 
\end{proof}

Evidently, Theorem \ref{SD-weak} is weaker than Theorem \ref{SD}. On the other hand, if $f=\tau_\alpha$, then \eqref{f-alpha} holds for arbitrarily large $A$, so that we can take $\ell$ to be arbitrarily large in \eqref{asymp-weak} and \eqref{asymp-weak2}. For general $f$, we may write $f=\tau_\alpha*f_0$. The partial sums of $\tau_\alpha$ can be estimated to arbitrary precision using \eqref{asymp-weak} with $\ell$ as large as we want. On the other hand, $f_0$ satisfies \eqref{f-alpha} with $\alpha=0$. So if we knew Theorem \ref{SD} in the special case when $\alpha=0$, we would deduce it in the case of $\alpha\neq0$ too (with a slightly weaker error term, as we will see). The next section fills in the missing step.

\section{The case $\alpha=0$ of Theorem \ref{SD}}\label{sec0}

\begin{thm}\label{SD-0}
Let $f$ be a multiplicative function with $|f|\le\tau_k$ and 
\eq{alpha=0}{
\sum_{p\le x}f(p)\log p \ll \frac{x}{(\log x)^A} 
}
for some $A>0$. Then
\[
\sum_{n\le x} f(n) \ll x(\log x)^{k-1-A} .
\]
The implied constant depend at most on $k$, $A$ and the implicit constant in \eqref{alpha=0}.
\end{thm}

\begin{proof}  As we discussed in Section \ref{aux}, we may assume that $f=\tau_f$. Our goal is to show the existence of an absolute constant $M$ such that
\eq{ind-hyp}{
\abs{\sum_{n\le x} f(n)} \le Mx (\log x)^{k-1-A} \quad(x\ge2) .
}
We argue by induction on the dyadic interval on which $x$ lies: if $x\le 2^{j_0}$, where $j_0$ is a large integer to be selected later, then \eqref{ind-hyp} holds by taking $M$ large enough in terms of $j_0$ (and $k$). Assume now that \eqref{ind-hyp} holds for all $x\le 2^j$ with $j\ge j_0$, and consider $x\in[2^{j/2},2^{j+1}]$. If $\eps=2/j_0$, then
\al{
\sum_{n\le x}f(n)\log n 
	&= \sum_{ab\le x}\Lambda_f(a)f(b)  \nn
	&= \sum_{2\le a\le x^\eps} \Lambda_f(a)\sum_{b\le x/a} f(b)
		+ \sum_{b\le x^{1-\eps}} f(b) \sum_{x^\eps<a\le x/b} \Lambda_f(a) \label{ind-e1},
}
where the restriction $a\ge2$ is automatic by the fact that $\Lambda_f$ is supported on prime powers. We may thus estimate the first sum in \eqref{ind-e1} by the induction hypothesis, and the second sum by \eqref{alpha=0}. Hence
\[
\sum_{n\le x}f(n)\log n 
	\ll \sum_{a\le x^\eps} |\Lambda_f(a)| \frac{Mx}{a}\cdot (\log(x/a))^{k-1-A} 
		+ \sum_{b\le x^{1-\eps}} |f(b)| \cdot \frac{x}{b(\log(x/b))^A} .
\]
The implied constant here and below depends on $k$, $A$ and the implied constant in \eqref{alpha=0}, but not on our choice of $M$. Since $|\Lambda_f|\le k\Lambda$ (and thus $|f|\le\tau_k$), as well as $(\log(x/b))^{-A}\le(\eps\log x)^{-A}$ for $b\le x^{1-\eps}$, we deduce that
\als{
\sum_{n\le x}f(n)\log n 
	&\ll Mx(\log x)^{k-1-A} \sum_{a\le x^\eps} \frac{\Lambda(a)}{a} 
		+ \frac{x}{(\eps\log x)^{A}} \sum_{b\le x^{1-\eps}} \frac{\tau_k(b)}{b} \\
	&\ll (\eps M+\eps^{-A}) x(\log x)^{k-A} 
}
uniformly for $x\in[2^{j/2},2^{j+1}]$. By partial summation, we thus conclude that 
\[
\sum_{n\le x}f(n)  = O(\sqrt{x}(\log x)^{k-1}) + \int_{\sqrt{x}}^x \frac{1}{\log y} \dee \sum_{n\le y}f(n)\log n 
	 \ll (\eps M+\eps^{-A}) x(\log x)^{k-1-A} 
\]
for $x\in(2^j,2^{j+1}]$. To complete the inductive step, we take $j_0=2/\eps$ to be large enough so as to make the $\ll \eps M$ part of the upper bound $\leq M/2$, and then $M$ to be large enough in terms of $j_0$ so that the $\ll \eps^{-A}$ part of the upper bound is also $\le M/2$. The theorem is thus proven.
\end{proof}

By Theorem \ref{SD-0} and the discussion in the last paragraph of Section \ref{perron}, we obtain Theorem \ref{SD} with the error term being $O(x(\log x)^{k+2|\alpha|-A-1}\log\log x)$. The reason for this weaker error term is that for the function $f_0=f*\tau_{-\alpha}$ we only know that $|\Lambda_{f_0}|\le (k+|\alpha|)\Lambda$. To deduce Theorem \ref{SD} in the stated form, we will modify the proof of Theorem \ref{SD-0} to handle functions $f$ satisfying \eqref{f-alpha} for general $\alpha$. This is accomplished in the next section. 

\section{Proof of Theorem \ref{SD}}\label{proof}

We introduce the auxiliary functions
\[
g(y) := {\bf 1}_{y>1} \cdot \sum_{j=0}^J \frac{c_j}{\Gamma(\alpha-j)} (\log y)^{\alpha-1-j} 
\quad\text{and}\quad
d(n) = f(n)-g(n) .
\]
Our goal is to show that 
\eq{SD-goal}{
\sum_{n\le x} d(n) \ll x(\log x)^{k-1-A} (\log\log x)^{{\bf1}_{J=A+1}} .
}
Theorem \ref{SD} then readily follows, since partial summation implies that 
\[
\sum_{n\le x}g(n) = \sum_{j=0}^J \frac{c_j}{\Gamma(\alpha-j)}\int_2^x (\log y)^{\alpha-1-j} \dee y 
	+ O(1+(\log x)^{\text{Re}(\alpha)-1}) .
\]

We start by showing a weak version of \eqref{SD-goal} for smoothened averages of $d$:


\begin{lem}\label{SD-lem2}
Let $f$ be a multiplicative function such that $f=\tau_f$ and for which \eqref{f-alpha} holds. Let $\psi:\R\to\R$ be a function in the class $C^\infty(\R)$ supported in $[\gamma,\delta]$ with $0<\gamma<\delta<\infty$. There are integers $J_1$ and $J_2$ depending at most on $A$ and $k$ such that
\[
\sum_{n=1}^\infty \frac{d(n)}{n}\, \psi\left(\frac{\log n}{\log x}\right) \ll 
	(1+\gamma^{-1})^{J_1} e^\delta
		\max_{j\le J_2} \|\psi^{(j)}\|_\infty \cdot  (\log x)^{\Re(\alpha)-A} (\log\log x)^{{\bf1}_{A=J+1}},
\]
for $x\ge2$, with the implied constant depending on $A$, $k$ and the implicit constant in \eqref{f-alpha}, but not on $\psi$.
\end{lem}


\begin{proof} All implied constants might depend on $A$, $k$ and the implicit constant in \eqref{f-alpha} without further notice. We will prove the lemma with $J_2=1+k+\fl{(A+2k)(J+2)/A}$ and $J_1=J_2+m$, where $m=J+k+1$. 

Set $\phi(y) = \psi(y)/y^m$ and note that
\[
\|\phi^{(j)}\|_\infty \ll_j (1+\gamma^{-1})^{j+m} \max_{0\le \ell\le j}\|\psi^{(\ell)}\|_\infty .
\]
It thus suffices to prove that
\eq{phi-SD}{
\sum_{n=1}^\infty \frac{d(n)(\log n)^m}{n}\, \phi\left(\frac{\log n}{\log x}\right) \ll  
	M \cdot (\log x)^{m+\Re(\alpha)-A} (\log\log x)^{{\bf1}_{A=J+1}} ,
}
where
\[
M:= e^\delta
		\max_{j\le J_2} \|\phi^{(j)}\|_\infty  .
\]

We consider the Mellin transform of the function $y\to \phi(\log y/\log x)$, that is to say the function
\[
\hat{\phi}_x(s) 
	:= \int_0^\infty  \phi\left(\frac{\log y}{\log x}\right) y^{s-1} \dee y
	= (\log x) \int_\gamma^\delta \phi(u) x^{su}\dee u .
\]
We then have that
\eq{f-g mellin}{
\sum_{n=1}^\infty \frac{d(n)(\log n)^m}{n}\, \phi\left(\frac{\log n}{\log x}\right) 
	= \frac{(-1)^m}{2\pi i} \int_{\sigma=1/\log x} 
			 D^{(m)}(s+1) \hat{\phi}_x(s) \dee s,
}
where $D:=F-G$ with $G(s) := \sum_n g(n)/n^s$.

We first bound $\hat{\phi}_x(s)$. We have the trivial bound 
\[
\hat{\phi}_x(s)  \ll e^\delta \|\phi\|_\infty \log x  \quad\text{when}\quad \Re(s)=1/\log x .
\]
Moreover, if we integrate by parts $j$ times in $\int_\gamma^\delta \phi(u) x^{su}\dee u$, we deduce that
\[
\hat{\phi}_x(s) 
	= \frac{\log x}{(-s\log x)^j} \int_\gamma^\delta \phi^{(j)}(u) x^{su}\dee u 
		\ll \frac{e^\delta \|\phi^{(j)}\|_\infty}{|s|^j(\log x)^{j-1}}  \quad\text{when}\quad \Re(s)=1/\log x\  ;
\]
we used here our assumption that $\text{supp}(\phi)\subset[\gamma,\delta]$, which implies that $\phi^{(j)}(u)=0$ for all $j$ and all $u\notin(\gamma,\delta)$. Putting together the above estimates, we conclude that
\eq{phi-IBP}{
\hat{\phi}_x(1/\log x+it)  \ll M \cdot \frac{\log x}{(1+|t|\log x)^j}  
}
for each $j\in\Z\cap[0,J_2]$, where the implied constant is independent of $\phi$.

Next, we bound $D^{(m)}(s+1)$ on the line $\Re(s)=1/\log x$. Since $d(n)(\log n)^m \ll \tau_k(n)(\log n)^m+(\log n)^{\Re(\alpha)-1+m}$ and $\Re(\alpha)\le k$, we conclude that $D^{(m)}(1+1/\log x+it) \ll  (\log x)^{k+m}$. Together with \eqref{phi-IBP} applied with $j=1+\fl{(A+2k)(J+2)/A}=J_2-k$, this bound implies that the integrand in the right hand side of \eqref{f-g mellin} is $\ll M \cdot (\log x)^{m+k}  \cdot (\log x)/(|t|\log x)^{J_2-k}$. Hence the portion of the integral with $|t|\ge (\log x)^{\frac{A}{J+2}-1}$ in \eqref{f-g mellin} contributes 
\[
\ll  M\cdot (\log x)^{m+k-\frac{A}{J+2}(J_2-k-1)} \le M\cdot (\log x)^{m-k-A}\le M\cdot(\log x)^{m+\Re(\alpha)-A} .
\] 

Finally, we bound the portion of the integral in \eqref{f-g mellin} with $|t|\le (\log x)^{\frac{A}{J+2}-1}$. Note that
\[
G^{(m)}(s+1) =(-1)^m \sum_{j=0}^J\frac{c_j}{\Gamma(\alpha-j)}  \sum_{n=2}^\infty \frac{(\log n)^{m+\alpha-1-j}}{n^{s+1}} .
\]
Since we have assumed that $m=J+k+1\ge J+|\alpha|+1$, and here we have that $|t|\le 1$ and $\sigma=1/\log x$, partial summation implies that
\als{
G^{(m)}(s+1) 
	&= (-1)^m\sum_{j=0}^J\frac{c_j}{\Gamma(\alpha-j)} \int_1^\infty \frac{(\log y)^{m+\alpha-1-j}}{y^{s+1}} \dee y 
	+ O(1) \\
	&= (-1)^{m+J} \sum_{j=0}^J c_j \frac{\Gamma(m+\alpha-j)}{\Gamma(\alpha-j)} s^{j-\alpha-m}
	+ O(1) \\
	&= (-1)^J G_{J+1}^{(m)}(s+1) + O(1) 
}
in the notation of Section \ref{aux}, where we used \eqref{gamma} with $s$ replaced by $s+1$ to obtain the second equality. 
  
We will apply Lemma \ref{L-bound}(b) with $s=1+1/\log x+it$. Notice that we have $|t|\le(\log x)^{\frac{A}{J+2}-1}\ll (\log x)^{\frac{A}{J+1}-1}/\log\log x$, so that the hypotheses of Lemma \ref{L-bound}(b) are met. Consequently, 
\eq{G-bound}{
D^{(m)}(1+1/\log x +it) 
	&=E_{J+1}^{(m)}(1+1/\log x+it) +O(1) \\
	&\ll (1+|t|\log x)^{J+1-\Re(\alpha)} (\log x)^{m-A+\Re(\alpha)}  .
}
Since $J_2\ge J+k+3$, relation \eqref{phi-IBP} with $j=J+k+3$ implies that
\[
\hat{\phi}_x(1/\log x+it) \ll M \cdot \frac{\log x}{(1+|t|\log x)^{J+k+3}} .
\]
We conclude that the portion of the integral with $|t|\le(\log x)^{\frac{A}{J+2}-1}$ in \eqref{f-g mellin} contributes $\ll M\cdot (\log x)^{m+\Re(\alpha)-A}$. This completes the proof of the lemma.
\end{proof}


We have that $f\log=f*\Lambda_f$. Since $\sum_{n=1}^\infty g(n)/n^s$ approximates the analytic behaviour of $F$, we might expect that the function 
\eq{f-g-d}{
g\log -g*\Lambda_f = d*\Lambda_f-d\log
}
is small on average. In reality, its asymptotic behaviour is a bit more complicated:


\begin{lem}\label{SD-lem3}
Let $f$ be a multiplicative function such that $f=\tau_f$ and for which \eqref{f-alpha} holds. There is a constant $\kappa\in\R$ such that 
\[
\sum_{n\le x}((\Lambda_f*g)(n)-g(n)\log n) = 
	\kappa x + O(x(\log x)^{k-A}(\log\log x)^{{\bf1}_{A=J+1}}) .
\]
The implied constant depend at most on $k$, $A$ and the implicit constant in \eqref{f-alpha}. The dependence on $A$ comes from both its size, and its distance from the nearest integer.
\end{lem}

\begin{proof} Set $h=\Lambda_f*g-g\log $. We begin by showing that there are coefficients $\kappa,\kappa_0,\kappa_1,\dots$ such that
\eq{h-ps}{
\sum_{n\le x} h(n) = \kappa x +
	\int_2^x \sum_{\substack{0\le j\le J \\ j\neq \alpha}} \frac{\kappa_j}{\Gamma(\alpha-j+1)} (\log y)^{\alpha-j}   \dee y 
		+ O(x(\log x)^{k-A}(\log\log x)^{{\bf1}_{A=J+1}}) .
}
We will later show by a different argument that the coefficients $\kappa_j/\Gamma(\alpha-j+1)$ with $j\neq\alpha$ must vanish. 

Partial summation implies that
\eq{h-e1}{
\sum_{n\le x}g(n)(\log n)^m = \int_2^x \sum_{\substack{0\le j\le J\\ j\neq \alpha}}
	\frac{c_j}{\Gamma(\alpha-j)} (\log y)^{\alpha-j+m-1} \dee y+O(1+(\log x)^{\Re(\alpha)+m-1}) ,
}
 as well as that
\eq{h-e2}{
\sum_{n\le x}g(n) 
	&= x \sum_{\substack{0\le j < J \\ j\neq \alpha}} \frac{\tilde{c}_j}{\Gamma(\alpha-j)} (\log x)^{\alpha-j-1} + O(x(\log x)^{\Re(\alpha)-J-2}) \\
	&=: x \tilde{g}(\log x) + O(x(\log x)^{k-1-A}) 
}
where the terms with $j=\alpha$ can be trivially excluded because $1/\Gamma(0)=0$, and we used that $J+1\ge A$ and $\Re(\alpha)\le k$.

We apply Dirichlet's hyperbola method to the partial sums of $\Lambda_f*g$ to find that
\[
\sum_{n\le x} (\Lambda_f*g)(n)
 = \sum_{b\le \sqrt{x}} g(b) \sum_{a\le x/b}\Lambda_f(a) 
		+ \sum_{a\le \sqrt{x}} \Lambda_f(a)  \sum_{b\le x/a} g(b) 
			- \sum_{a\le\sqrt{x}}\Lambda_f(a)\sum_{b\le \sqrt{x}}g(b) .
\]
We then insert relations \eqref{f-alpha} and \eqref{h-e2} to deduce that
\als{
\frac{1}{x}\sum_{n\le x} (\Lambda_f*g)(n)
	&=  \alpha \sum_{b\le \sqrt{x}} \frac{g(b)}{b}
		+ \sum_{a\le \sqrt{x}} \frac{\Lambda_f(a)\tilde{g}(\log(x/a))}{a} 
			-\alpha  \tilde{g}\left(\frac{\log x}{2}\right)
			+ E,
}
where
\als{
E &\ll  (\log x)^{-A} \sum_{b\le \sqrt{x}} \frac{|g(b)|}{b} 
	+ (\log x)^{k-1-A} \sum_{a\le \sqrt{x}} \frac{|\Lambda_f(a)|}{a} 
	+ \frac{| \tilde{g}(\log\sqrt{x})|}{(\log x)^A}  + (\log x)^{\Re(\alpha)-J-1} \\
	&\ll (\log x)^{-A} \sum_{b\le \sqrt{x}} \frac{(\log b)^{k-1}}{b} + (\log x)^{k-A}
	\ll (\log x)^{k-A} .
}
Consequently, 
\als{
\frac{1}{x}\sum_{n\le x} (\Lambda_f*g)(n)
	&=  \alpha \sum_{b\le \sqrt{x}} \frac{g(b)}{b}
		+ \sum_{a\le \sqrt{x}} \frac{\Lambda_f(a)\tilde{g}(\log(x/a))}{a} 
			-\alpha  \tilde{g}\left(\frac{\log x}{2}\right)
			+ O((\log x)^{k-A}) .
}
For the sum of $g(b)/b$, we use the Euler-McLaurin summation formula to find that
\als{
\sum_{b\le \sqrt{x}} \frac{g(b)}{b}
	&= \int_2^{\sqrt{x}} \frac{g(y)}{y} \dee y 
		  + \int_2^{\sqrt{x}} \{y\}(g'(y)/y-g(y)/y^2) \dee y +  O((\log x)^{\Re(\alpha)-1}/\sqrt{x})  \\
	&= \sum_{\substack{0\le j\le J \\ j\neq\alpha}} 
		\frac{ {c}_j}{\Gamma(\alpha-j+1)} \cdot \frac{(\log x)^{\alpha-j}}{2^{\alpha-j}} 
		+c  +  O((\log x)^{\Re(\alpha)-1}/\sqrt{x})  ,
}
where
\[
c :=   - \sum_{\substack{0\le j\le J \\ j\neq\alpha}} 
		\frac{ {c}_j\cdot (\log 2)^{\alpha-j}}{\Gamma(\alpha-j+1)} 
				+ \int_2^\infty  \{y\}(g'(y)/y-g(y)/y^2) \dee y  .
\]
 
It remains to estimate the sum over $a$. By partial summation and \eqref{f-alpha}, we find that
\als{
\sum_{a\le \sqrt{x}} \frac{\Lambda_f(a)\tilde{g}(\log(x/a))}{a} 
	&=\alpha \int_1^{\sqrt{x}} \frac{\tilde{g}(\log(x/y))}{y} \dee y
		+\alpha\tilde{g}(\log x) + O((\log x)^{\Re(\alpha)-1-A}) \\
	&\quad + \int_1^{\sqrt{x}} \frac{R(y)q(\log(x/y))}{y^2}\dee y ,
}
where $R(y):=\sum_{n\le y}\Lambda_f(n)-\alpha y\ll y(\log y)^{-A}$ and
\[
q(y):=\tilde{g}(y)+\tilde{g}\,'(y) 
	=\sum_{j=0}^J \frac{c_jy^{\alpha-j-1}}{\Gamma(\alpha-j)} + 
\frac{\tilde{c}_Jy^{\alpha-J-2}}{\Gamma(\alpha-J-1)} 
	= g(e^y) + \frac{\tilde{c}_J y^{\alpha-J-2}}{\Gamma(\alpha-J-1)} 
\]
using the fact that $c_j=\tilde{c}_j+\tilde{c}_{j-1}$. In the main term, we make the change of variables $t=\log(x/y)$. In the error term, we develop $q$ into Taylor series about $\log x$: we have that
\[
q(\log(x/y)) =  \sum_{j=0}^{J-1} \frac{q^{(j)}(\log x)}{j!}  (-\log y)^j 
	+ O((\log x)^{\Re(\alpha)-J-1}(\log y)^J) 
\]
for $y\le\sqrt{x}$. Since  $\int_2^{\sqrt{x}} (\log y)^{J-A} y^{-1}\dee y \ll (\log x)^{J+1-A} (\log\log x)^{{\bf1}_{A=J+1}}$ by our assumption that $J<A\le J+1$, we thus find that
\als{
\sum_{a\le \sqrt{x}} \frac{\Lambda_f(a)\tilde{g}(\log(x/a))}{a} 
	&=\alpha \int_{\frac{\log x}{2}}^{\log x} \tilde{g}(t)\dee t+\alpha\tilde{g}(\log x) 
	+\sum_{j=0}^{J-1} \frac{q^{(j)}(\log x)}{j!} 
		\int_1^{\sqrt{x}} \frac{R(y)(-\log y)^j}{y^2}\dee y \\
	&\qquad  + O((\log x)^{\Re(\alpha)-A}(\log\log x)^{{\bf1}_{A=J+1}}) .
}
The first two terms on the right hand side of this last displayed equation can be computed exactly: they equal
\als{
&\alpha \sum_{\substack{0\le j\le J \\ j\neq\alpha}} \frac{\tilde{c}_j\cdot (1-2^{-\alpha+j})}{\Gamma(\alpha-j+1)}(\log x)^{\alpha-j} 
	+\alpha \sum_{\substack{0\le j\le J \\ j\neq\alpha}} \frac{\tilde{c}_j}{\Gamma(\alpha-j)}(\log x)^{\alpha-j-1} \\
&\quad= 	
	\alpha \sum_{\substack{0\le j\le J \\ j\neq\alpha}} 
		\frac{c_j\cdot (1-2^{-\alpha+j})}{\Gamma(\alpha-j+1)}(\log x)^{\alpha-j} 
	+\frac{\alpha \tilde{c}_J\cdot(1-2^{-\alpha+J+1})}{\Gamma(\alpha-J)} (\log x)^{\alpha-J-1}
	+ \alpha \tilde{g}\left(\frac{\log x}{2}\right) ,
}
since $c_j=\tilde{c}_j+\tilde{c}_{j-1}$.  Using the estimates $q^{(j)}(\log x) \ll (\log x)^{\Re(\alpha)-j-1}$ and
\als{
\int_1^{\sqrt{x}} \frac{R(y)(-\log y)^j}{y^2}\dee y
	&= \int_1^\infty \frac{R(y)(-\log y)^j}{y^2}\dee y + O((\log x)^{j-A+1})  \\
	&=: I_j+ O((\log x)^{j-A+1}) 
}
for $j\le J-1<A-1$, we conclude that
\als{
\sum_{a\le \sqrt{x}} \frac{\Lambda_f(a)\tilde{g}(\log(x/a))}{a} - \alpha \tilde{g}\left(\frac{\log x}{2}\right)
	&= \alpha \sum_{\substack{0\le j\le J \\ j\neq\alpha}} \frac{c_j\cdot (1-2^{-\alpha+j})}{\Gamma(\alpha-j+1)}(\log x)^{\alpha-j} +\sum_{j=0}^{J-1} I_j\cdot  \frac{q^{(j)}(\log x)}{j!}
	\\
	&\quad	
	  + O((\log x)^{\Re(\alpha)-A}(\log\log x)^{{\bf1}_{A=J+1}})) .
}
Putting together the above estimates yields the formula
\[
\sum_{n\le x} h(n) = \kappa x +
	  \sum_{\substack{0\le j\le J \\ j\neq \alpha}} \frac{\tilde{\kappa}_j}{\Gamma(\alpha-j+1)}  x (\log x)^{\alpha-j}   
	+	O((\log x)^{k-A}(\log\log x)^{{\bf1}_{A=J+1}}) .
\]
where $\kappa$ and $\tilde{\kappa}_j$ are some constants that can be explicitly computed in terms of the constants $c$, $c_j$ and $I_j$. We may then write the above formula in the form \eqref{h-ps} using the fact that
\[
\frac{x(\log x)^\beta}{\Gamma(\beta+1)} 
	= \int_2^x \frac{(\log y)^\beta}{\Gamma(\beta+1)} \dee y
	+\int_2^x \frac{(\log y)^{\beta-1}}{\Gamma(\beta)} \dee y +O(1)  ,
\]
thus completing the proof of \eqref{h-ps}. 
 
To complete the proof of the lemma, we will show that $\kappa_j/\Gamma(\alpha-j+1)=0$ for all $j\neq\alpha$ with $j<A-k+\Re(\alpha)$. To see this, let $\psi$ be a smooth test function such that
\[
\begin{cases}
\psi(u)=1 &\text{if}\ u\in[0.7,0.9],\\
\psi(u)\in[0,1]&\text{if}\ u\in[2/3,1]\setminus[0.7,0.9],\\
\psi(u)=0 &\text{otherwise},
\end{cases}
\]
and set 
\[
L(x):= \frac{1}{\log x} \sum_n \frac{h(n)}{n} \psi\left(\frac{\log n}{\log x}\right).
\]
We calculate $L(x)$ in two different ways. 

On the one hand, partial summation and \eqref{h-ps} imply that
\eq{L-firstway}{
L(x) &= \kappa\int_0^\infty \psi(t)\dee t  
	+ \sum_{\substack{0\le j\le J \\ \alpha\neq j}} \frac{\kappa_j\cdot (\log x)^{\alpha-j} }{\Gamma(\alpha-j+1)} \int_0^\infty  \psi(t) t^{\alpha-j}\dee t \\
	&\qquad +O((\log x)^{k-A}(\log\log x)^{{\bf1}_{A=J+1}}) .
}

On the other hand, we have that $h=d\log - \Lambda_f*d$ by \eqref{f-g-d}. An application of Lemma \ref{SD-lem2} yields that
\[
L(x)  = O((\log x)^{\Re(\alpha)-A}(\log\log x)^{{\bf1}_{A=J+1}}) 
	-\frac{1}{\log x}  \sum_{a,b} \frac{d(a)\Lambda_f(b)}{ab} \psi\left(\frac{\log(ab)}{\log x}\right) .
\]
Then we observe that, for each fixed $b\le \sqrt{x}$, the function $u\to \psi(u+\log b/\log x)$ is smooth and supported in $[1/6,1]$. We re-apply Lemma \ref{SD-lem2} to find that
\[
 \frac{1}{\log x} \sum_{ b\le\sqrt{x}} \frac{ \Lambda_f(b)}{b} 
		\sum_{a} \frac{d(a) }{a} \psi\left(\frac{\log(ab)}{\log x}\right)  = O((\log x)^{\Re(\alpha)-A}(\log\log x)^{{\bf1}_{A=J+1}}).
 \]
Finally, for fixed $a\le\sqrt{x}$, we use relation \eqref{f-alpha} to find that
\als{
 \frac{1}{\log x}\sum_{b\ge\sqrt{x}} \frac{\Lambda_f(b)}{b} \psi\left(\frac{\log(ab)}{\log x}\right) 
	&=  \frac{\alpha}{\log x}\int_{\sqrt{x}}^\infty  \psi\left(\frac{\log(ay)}{\log x}\right) \frac{\dee y}{y} + O((\log x)^{-A}) \\
	&=\alpha \int_{\frac{1}{2}+\frac{\log a}{\log x}}^\infty \psi(t)\dee t+ O((\log x)^{-A}) .
}
We thus conclude that
\[
L(x) = -\alpha  \sum_a \frac{d(a)}{a} \int_{\frac{1}{2}+\frac{\log a}{\log x}}^\infty \psi(t)\dee t 
	+ O((\log x)^{k-A}(\log\log x)^{{\bf1}_{A=J+1}}) .
\]
The function
\[
\Psi(u) := \int_{1/2+u}^\infty \psi(t)\dee t
\]
is a smooth function supported in $[0,1/2]$ and that is constant for $u\le 1/6$. Hence the function $\phi(u):=\Psi(2u)-\Psi(u)$ is supported on $[1/12,1/2]$. Lemma \ref{SD-lem2} then implies that
\als{
L(x) - L(\sqrt{x})
	&= \alpha  \sum_a \frac{d(a)}{a} \phi\left( \frac{\log a}{\log x}\right)
		+ O((\log x)^{k-A}(\log\log x)^{{\bf1}_{A=J+1}}) \\
	& \ll (\log x)^{k-A}(\log\log x)^{{\bf1}_{A=J+1}}  .
}
By our choice of $\psi$, comparing the above estimate with \eqref{L-firstway} proves that $\kappa_j/\Gamma(\alpha-j+1)=0$ for all $j\neq \alpha$ with $j<A-k+\Re(\alpha)$, and the lemma follows.  
\end{proof}

We are finally ready to prove our main result:

\begin{proof}[Proof of Theorem \ref{SD}]
We will prove that there is some constant $M$ such that
\eq{SD-indhyp}{
\abs{\sum_{n\le x} d(n) }  \le Mx(\log x)^{k-1-A}(\log\log x)^{{\bf1}_{A=J+1}}  
}
for all $x\ge2$. Together with \eqref{h-e1} and \eqref{h-e2}, this will immediately imply Theorem \ref{SD}. 

As in the proof of Theorem \ref{SD-0}, we induct on the dyadic interval in which $x$ lies. We fix some large integer $j_0$ and note that \eqref{SD-indhyp} is trivially true when $2\le x\le 2^{j_0}$ by adjusting the constant $M$. Fix now some integer $j\ge j_0$ and assume that \eqref{SD-indhyp} holds when $2\le x\le 2^j$. We want to prove that \eqref{SD-indhyp} also holds for $x\in[2,2^{j+1}]$. Whenever we use a big-Oh symbol, the implied constant will be independent of the constant $M$ in \eqref{SD-indhyp}.

Let $x\in[2^{j(1-\eps)},2^{j+1}]$ and $\eps=2/j_0$. We have that 
\[
d\log
	= f*\Lambda_f - g\log
	= d*\Lambda_f+ h
\]
with $h:=g*\Lambda_f-g\log$. Applying Lemma \ref{SD-lem3}, we find that
\als{
\sum_{n\le x}d(n)\log n 
	&= \sum_{ab\le x} \Lambda_f(a) d(b) + \kappa x + O(x(\log x)^{k-A}(\log\log x)^{{\bf1}_{A=J+1}}) \\
	&= \sum_{b\le x^{1-\eps}}d(b) \sum_{a\le x/b}\Lambda_f(a)
		+ \sum_{2\le a\le x^\eps} \Lambda_f(a) \sum_{x^{1-\eps}<b\le x/a}  d(b)  \\
	&\qquad + \kappa x + O(x(\log x)^{k-A}(\log\log x)^{{\bf1}_{A=J+1}}) .
}
We estimate the sum $\sum_{a\le x/b}\Lambda_f(a)$ by \eqref{f-alpha}, and the sum $\sum_{b\le x/a}d(b)$ by the induction hypothesis, since $a\ge2$ here. As in the proof of Theorem \ref{SD-0}, and using the bound $|d(b)|\le |f(b)|+|g(b)|\ll \tau_k(b)+(\log b)^{k-1}$, we conclude that
\als{
\sum_{n\le x}d(n)\log n 
	&=\alpha x \sum_{b\le x^{1-\eps}} \frac{d(b)}{b} 	
		+ \kappa x + O(x(\log x)^{k-A}(\log\log x)^{{\bf1}_{A=J+1}}) \\
	&\qquad + O\left(x\sum_{b\le x^{1-\eps}}\frac{|d(b)|}{b\log^A(x/b)} 
			+ \frac{Mx(\log\log x)^{{\bf1}_{A=J+1}}}{(\log x)^{A+1-k}} \sum_{2\le a\le x^\eps}\frac{|\Lambda_f(a)|}{a} \right) \\
	&=\alpha x \sum_{b\le x^{1-\eps}} \frac{d(b)}{b} 
		+\kappa x +O((\eps^{-A}+\eps M)x(\log x)^{k-A}(\log\log x)^{{\bf1}_{A=J+1}}) )  
}
for all $x\in[2^{j(1-\eps)},2^{j+1}]$. If we could show that the main terms cancel each other, then the induction would be completed as in Theorem \ref{SD-0}. To show this, we will use Lemma \ref{SD-lem2}. 

Firstly, note that when $x\in[2^{j(1-\eps)},2^{j+1}]$, we have that $x^\eps\ge2$, so that $x^{1-\eps}\le x/2\le 2^j$. Re-applying the induction hypothesis yields the bound
\[
\sum_{x^{1-\eps}<b\le 2^j} \frac{d(b)}{b} 
	\ll \eps M (\log x)^{k-A} (\log\log x)^{{\bf1}_{A=J+1}}.
\]
Setting $\lambda_j= \kappa +\alpha \sum_{b\le 2^j} d(b)/b$ then implies that
\eq{E-almost}{
\sum_{n\le x} d(n)\log n 
	= \lambda_j x
		+O(x(\eps^{-A}+\eps M) (\log x)^{k-A}(\log\log x)^{{\bf1}_{A=J+1}}) 
}
for all $x\in[2^{j(1-\eps)},2^{j+1}]$. Set $X=2^j$ and let $\psi$ be a smooth function that is non-negative, supported on $[1-\eps,1]$, assumes the value 1 on $[1-\eps/2,1-\eps/3]$, and for which $\|\psi^{(j)}\|_\infty \ll_j \eps^{-j}$ for all $j$. Then Lemma \ref{SD-lem2} gives us that
\[
\sum_{n=1}^\infty \frac{d(n)}{n} \, \psi\left(\frac{\log n}{\log X}\right) 
	\ll \eps^{-J_2} (\log X)^{k-A}  (\log\log x)^{{\bf1}_{A=J+1}}
\]
for some $J_2=J_2(k,A)>A$. On the other hand, if we set $\phi(u)=\psi(u)/u$ and $R(x)=\sum_{n\le x}d(n)\log n-\lambda_jx$, then partial summation and \eqref{E-almost} yield that
\als{
\sum_{n=1}^\infty \frac{d(n)}{n} \, \psi\left(\frac{\log n}{\log X}\right)
	&=\frac{1}{\log X} \sum_{n=1}^\infty \frac{d(n)\log n}{n} \, \phi\left(\frac{\log n}{\log X}\right) \\
	&=\lambda_j \int_1^\infty \frac{\phi(\frac{\log y}{\log X})}{y\log X}\dee y  
		+\int_{X^{1-\eps}}^X 
			\left(\frac{\phi(\frac{\log y}{\log X})}{\log X}
				-\frac{\phi'(\frac{\log y}{\log X})}{\log^2X} \right)
			\frac{R(y)}{y^2} \dee y \\
	&=\lambda_j \int_0^\infty \phi(u)\dee u 
		+O\left(\eps (\eps^{-A}+\eps M) x(\log x)^{k-A}(\log\log x)^{{\bf1}_{A=J+1}}\right)  ,
}
since $\|\phi\|_\infty\ll1$ and $\|\phi'\|_\infty \ll \eps^{-1} \ll \log X$. Noticing that we also have that $\int_0^\infty \phi(u)\dee u \gg\eps$ by our choice of $\phi$, we deduce that 
\[
\lambda_j \ll (\eps^{-J_2}+\eps M)(\log X)^{k-A}(\log\log x)^{{\bf1}_{A=J+1}} ,
\] 
whence
\[
\sum_{n\le x} d(n)\log n \ll x  (\eps^{-J_2}+\eps M)(\log x)^{k-A}(\log\log x)^{{\bf1}_{A=J+1}} 
\]
for $x\in[2^{j(1-\eps)},2^{j+1}]$. We then apply partial summation to find that
\als{
\sum_{n\le x} d(n)
	&= O(x^{1-\eps}(\log x)^{k-1}) + \int_{x^{1-\eps}}^x \frac{1}{\log y} \dee \sum_{n\le y} d(n)\log n  \\
	&	\ll x  (\eps^{-J_2}+\eps M)(\log x)^{k-A-1} (\log\log x)^{{\bf1}_{A=J+1}}
}
for $x\in(2^j,2^{j+1}]$, since $x^{\eps}\gg (\eps\log x)^A$. Choosing $\eps$ to be small enough, and then $M$ to be large enough in terms of $\eps$, similarly to the proof of Theorem \ref{SD-0}, completes the inductive step. Theorem \ref{SD} then follows.
\end{proof}

\section{The error term  in Theorem \ref{SD} is necessary}\label{best-possible}
 
To obtain the specific shape of the error term in Theorem \ref{f-alpha}, we had to use increasingly complicated arguments. A natural question is whether one can produce a sharper error term. We will show that the error term in Theorem \ref{f-alpha} is optimal, when $\alpha$ is a non-negative real number and $A$ is not an integer. Precisely, we have the following result:

\begin{cor}\label{best possible2}
Let $\alpha=k$ and $A$ be  given real numbers with $\alpha \geq 1$, where $A>0$ is not an integer and let $J$ be the largest integer $<A$. There exists a multiplicative function $f$ satisfying \eqref{f-alpha} and the inequality $|f|\le\tau_k$, and coefficients   $\tilde{c}_j$   defined by \eqref{taylor-coeffs}, as well as $\gamma\neq 0$, such that
\[
 \sum_{n\le x} f(n)  
 	=  x \sum_{j=0}^{J} \tilde{c}_j  \frac{ (\log x)^{\alpha-j-1} } {\Gamma(\alpha-j)}
		+(\gamma+o_{x\to\infty}(1)) x (\log x)^{k-1-A}  .
\]
\end{cor}

This follows easily from the following theorem:

\begin{thm}\label{best possible}
Let $\alpha\in\C$ and $A>|\alpha|-\Re(\alpha)$. There exists a multiplicative function $f$ satisfying \eqref{f-alpha} and the inequality $|\Lambda_f|\le \max\{|\alpha|,1\}\Lambda$, and for which there exist coefficients $\beta_j$, $j<A$, and $\gamma\neq0$ such that
\[
 \sum_{n\le x} f(n)  
 	=  x \sum_{0\le j<A} \beta_j \frac{ (\log x)^{\alpha-j-1} } {\Gamma(\alpha-j)} 
		+(\gamma+o_{x\to\infty}(1)) x (\log x)^{\alpha-1-A}  .
\]
\end{thm}

\begin{rem}
We say a few words to explain our hypotheses in Corollary \ref{best possible2}. Comparing the result in Theorem \ref{best possible} with Theorem \ref{SD}, we see that each $\beta_j = \tilde{c}_j$. To ensure that the error term is as big as desired we need that $k=\text{Re}(\alpha)$ and so, since $|\alpha|\leq k$, this implies that $\alpha=k$ is a non-negative real number. To obtain that $|f|\le\tau_k$ we need that $|\Lambda_f|\le k\Lambda$ and so   $k\geq 1$. The term with exponent $\alpha-1-A$ is only not part of the series of terms with exponents 
$\alpha-j-1$ if $A$ is not an integer. This explains the assumptions in Corollary \ref{best possible2}.
\end{rem}

To construct $f$ in the proof of Theorem \ref{best possible}, we let $\theta=\text{arg}(\alpha)$, fix a parameter $\eps\in[0,1]$ that will be chosen later, and set
\[
f(p^\nu) = \binom{\alpha_p+\nu-1}{\nu} ,
\quad\text{where}\quad \alpha_p 
	=  \begin{cases}
		\alpha-e^{i\theta}(\log 2/\log p)^A &\text{if}\ p>2,\\
		\alpha -e^{i\theta}(1-\eps) &\text{if}\ p=2,
	\end{cases}
\]
that is to say $f$ is the multiplicative function with Dirichlet series $\prod_p (1-1/p^s)^{-\alpha_p}$. 
We have selected $\alpha_p$ so that it is  a real scalar multiple of $\alpha$, with $|\alpha_p|\leq \max\{|\alpha|,1\}$. Therefore
 $f$ satisfies \eqref{f-alpha}, as well as the inequality $|\Lambda_f|\le \max\{|\alpha|,1\}\Lambda$. We have the following key estimate:

\begin{lem}\label{g-ps} Write $f=\tau_\alpha*g$. There are constants $\lambda_j$ with $\lambda_0=-e^{i\theta} (\log2)^A\sum_{m=1}^\infty g(m)/m$ such that
\[
\sum_{n\le x}g(n) = \sum_{j=0}^J
	\frac{\lambda_j x}{(\log x)^{A+1+j}} + O\left( \frac{x(\log\log x)^{2A+1}}{(\log x)^{2A+1}}\right) .
\]
\end{lem}

\begin{proof} The Dirichlet series of $g$ is given by $\prod_p(1-1/p^s)^{\alpha-\alpha_p}$, whence
\[
g(p^\nu) = \binom{\alpha_p-\alpha+\nu-1}{\nu} .
\]
Since $|\alpha-\alpha_p|\le1$, we have that $|g|\le 1$. Note also that $g(p)\ll 1/(\log p)^A$, so that $\sum_{p,\ \nu\ge1}|g(p^\nu)|/p^\nu=O(1)$. By multiplicativity, we conclude that
\eq{g-series}{
\sum_{m=1}^\infty \frac{|g(m)|}{m} =O(1) .
}
In particular, this proves that $\lambda_0$ is well-defined.

To estimate the partial sums of $g$, we take $y:=x^{1/\log\log x}$ and decompose $n$ as $n=ab$, with $a$ having all its prime factors $\le y$ and $b$ having all its prime factors $>y$. Since $|g|\le 1$, the $n$'s with $b$ not being square-free contribute $\ll x/y$ to the sum $\sum_{n\le x}g(n)$, and the $n$'s with $b=1$ contribute 
\[
\le\#\{n\le x:p|n\ \Rightarrow p\le y\} \ll \frac{x}{(\log x)^{2A+1}} 
\] 
(cf. Corollary III.5.19 in \cite{Ten}). Similarly, the number of $n$'s with $a>\sqrt{x}$ contribute $\ll x/(\log x)^{2A+1}$. Finally, if $b$ is square-free with $\omega(b)\ge 2$, then we write $n=mpq$ with $p$ being the largest prime factor of $n$ and $q$ being its second largest prime factor, for which we know that $p,q>y$. We thus find that the contribution of such $n$ is
\als{
&\le \sum_{m\le x/y^2} |g(m)| \sum_{y<q\le \sqrt{x/m}} \frac{(\log2)^A}{(\log q)^A} 
	\sum_{q<p\le x/mq} \frac{(\log2)^A}{(\log p)^A}  \\
&\ll \sum_{m\le x/y^2} |g(m)| \sum_{y<q\le\sqrt{x/m}} \frac{1}{(\log q)^A} \cdot \frac{x/mq}{(\log q)^{A+1}} \\
&\ll \frac{x}{(\log y)^{2A+1}} \sum_{m\le x/y^2} \frac{|g(m)|}{m} \ll \frac{x}{(\log y)^{2A+1}} .
}
Consequently,
\eq{g-first bound}{
\sum_{n\le x}g(n) 
	= -e^{i\theta} (\log 2)^A \sum_{\substack{m\le \sqrt{x} \\ P^+(m)\le y}} g(m)
		\sum_{y<p\le x/m} \frac{1}{(\log p)^A} 
			+ O\left(\frac{x}{(\log y)^{2A+1}}\right) .
}
Before continuing, we note for future reference that the exact same argument can be applied with $|g|$ in place of $g$ and yield the estimate
\eq{g-second bound}{
\sum_{n\le x}|g(n)|
	&= (\log 2)^A \sum_{\substack{m\le \sqrt{x} \\ P^+(m)\le y}} |g(m)|
		\sum_{y<p\le x/m} \frac{1}{(\log p)^A} 
			+ O\left(\frac{x}{(\log y)^{2A+1}}\right) \\
	&\ll \sum_{m\le \sqrt{x}} |g(m)|\cdot \frac{x/m}{(\log x)^{A+1}} + \frac{x}{(\log y)^{2A+1}} 
	\ll \frac{x}{(\log x)^{A+1}} ,
}
where we used \eqref{g-series}.

Going back to estimating the partial sums of $g$, the sum over $p$ in relation \eqref{g-first bound} is $\int_y^{x/m} \dee t/(\log t)^{A+1}+O(x/(m(\log x)^{2A+1}))$ by the Prime Number Theorem. Integrating by parts, we thus have that 
\[
\sum_{y<p\le x/m} \frac{1}{(\log p)^A} 
	= \sum_{0\le j<A}\frac{d_jx/m}{(\log(x/m))^{A+1+j}} + O\left(\frac{x/m}{(\log x)^{2A+1}}\right) 
\]
for some constants $d_j$ with $d_0=1$. Finally, note that
\[
\frac{1}{(\log(x/m))^{A+1+j}}
	= \sum_{i=0}^{J-j} \binom{A+j+i}{i} \frac{(\log m)^i}{(\log x)^{A+i+j+1}} 
		+ O\left( \frac{(\log m)^{A-j}}{(\log x)^{2A+1}}\right) .
\]
Since $\sum_{m>\sqrt{x}} |g(m)|(\log m)^\ell/m \ll (\log x)^{\ell-A}$ for $\ell<A$, and $\sum_{m\le \sqrt{x}}|g(m)|(\log m)^A/m \ll \log\log x$, by \eqref{g-second bound} and partial summation, the lemma follows.
\end{proof}

Finally, we need the following lemma in order to calculate the main terms in Theorem \ref{best possible}.

\begin{lem}\label{tau_alpha(m)/m}
Fix $\alpha\in\C$ and $j\in\Z_{\ge0}$. For $x\ge2$, we have that
\[
\sum_{m\le x} \frac{\tau_\alpha(m)(\log m)^j}{m} =
	 \frac{(\log x)^{\alpha+j}}{(\alpha+j)\Gamma(\alpha)}+R
\]
where we interpret $\Gamma(\alpha)(\alpha+j)$ as $(-1)^j/j!$ when $\alpha=-j$ (i.e. the residue of $\Gamma$ at $-j$) and
\[
R \ll_\alpha 
	\begin{cases} 	
		1	&\text{if}\ -\Re(\alpha)<j<-\Re(\alpha)+1,\\
		(\log x)^{\Re(\alpha)+j-1} &\text{otherwise} .
		\end{cases}
\]
\end{lem}

\begin{proof} There is $c=O_\alpha(1)$ such that
\al{
\sum_{n\le x} \tau_\alpha(n) 
	&= \frac{x(\log x)^{\alpha-1}}{\Gamma(\alpha)} 
		+ \frac{cx(\log x)^{\alpha-2}}{\Gamma(\alpha-1)} 
	+ O(x(\log x)^{\Re(\alpha)-3}) \label{tau_alpha-selberg1}  \\
	&= \frac{x(\log x)^{\alpha-1}}{\Gamma(\alpha)} 
		+ O(x(\log x)^{\Re(\alpha)-2}) \label{tau_alpha-selberg}
			\quad(x\ge2) .
}
When $3/2>\Re(\alpha)+j>0$, the lemma follows by partial summation and \eqref{tau_alpha-selberg1}, whereas when $\Re(\alpha)+j>3/2$, we use \eqref{tau_alpha-selberg}. 

Next, when $\Re(\alpha)+j<0$, we note that the sum $\sum_{m=1}^\infty \tau_\alpha(m)(\log m)^j/m$ converges amd is equal to 0. Indeed, it equals $(-1)^j$ times the $j$-th derivative of $\zeta(s)^{\alpha}$ evaluated at $s\to1^+$, which tends to 0 in virtue of our hypothesis that $\Re(\alpha)+j<0$. Hence
\[
\sum_{m\le x} \frac{\tau_\alpha(m)(\log m)^j}{m} 
	= - \sum_{m>x} \frac{\tau_\alpha(m)(\log m)^j}{m} .
\]
Estimating the right hand side using \eqref{tau_alpha-selberg} and partial summation proves the lemma in this case too.

It remains to consider the lemma when $\Re(\alpha)=-j$. We then simply observe that
\[
\sum_{m\le x} \frac{\tau_\alpha(m)(\log m)^j}{m} 
 	= \lim_{\eps\to0^+} \sum_{m\le x} \frac{\tau_{\alpha-\eps}(m)(\log m)^j}{m} 
\]
and apply the case when $\Re(\alpha)<j$ proven above.
\end{proof}

We are now ready to estimate the partial sums of $f$:

\begin{proof}[Proof of Theorem \ref{best possible}] For the summatory function of $\tau_\alpha$, we already know an asymptotic series expansion: there exist constants $\kappa_0,\kappa_1,\ldots$ such that for any fixed $\ell\geq 1$, 
\al{
\sum_{n\leq x} \tau_\alpha(n) 
	= x\sum_{j=0}^{\ell-1} \kappa_j (\log x)^{\alpha-j-1} + O(x (\log x)^{\Re(\alpha)-\ell-1}) \label{tau_a-asymp} 
}
by \eqref{SD-classical}, or by Theorem \ref{SD-weak}, which we can apply for arbitrarily large $A$ when $f=\tau_\alpha$. We then estimate the partial sums of $f$ using the Dirichlet hyperbola method:
\eq{f-sharp-hyp}{
\sum_{n\le x}f(n) 
= \sum_{n\le x^\theta} g(n) \sum_{m\le x/n} \tau_\alpha(m) +\sum_{m\le x^\theta} \tau_\alpha(m)  \sum_{x^\theta<n\le x/m}  g(n) ,
}
where $\theta\in[1/3,2/3]$ is a parameter to be chosen in the end of the proof. Letting, as usual, $J$ to be the largest integer $<A$, and using relation \eqref{tau_a-asymp}, we find that
\[
 \sum_{n\le x^\theta} g(n) \sum_{m\le x/n} \tau_\alpha(m) 
 	= x \sum_{n\le x^\theta} \frac{g(n)} n \left(  \sum_{j=0}^J \kappa_j (\log(x/n))^{\alpha-j-1}
	 + O(  (\log x)^{\Re(\alpha)-J-2})\right) ,
\]
which equals
\[
x  \sum_{i+j=0}^J \kappa_j  \frac{\Gamma(j-\alpha+i+1)} {\Gamma(j-\alpha+1)i!}   (\log x)^{\alpha-j-i-1}  
\sum_{n\le x^\theta} \frac{g(n)} n     (\log n)^i + O(  x(\log x)^{\Re(\alpha)-J-2})  .
\]
Now, for $i\le J<A$, we have
\[
\sum_{n\le x^\theta} \frac{g(n)} n     (\log n)^i  = (-1)^i G^{(i)}(1) 	
	- \frac{\theta^{i-A}/(A-i)}{(\log x)^{A-i}} + O\left( \frac 1{(\log x)^{A-i+1}} \right) 
\]
by the Prime Number Theorem. Substituting in then gives
\[
\sum_{n\le x^\theta} g(n) \sum_{m\le x/n} \tau_\alpha(m) 
	= x   \sum_{v= 0}^J \beta_v (\log x)^{\alpha-v-1}   + cx(\log x)^{\alpha-A-1}+ O(  x(\log x)^{\Re(\alpha)-J-2}  )  ,
\]
where
\[
\beta_v=\sum_{i+j=v}
	 (-1)^i G^{(i)}(1) \kappa_j  \frac{\Gamma(j-\alpha+i+1)} {\Gamma(j-\alpha+1)i!}
\quad\text{and}\quad
c=-\kappa_0  \sum_{i=0}^J    \frac{\Gamma(-\alpha+i+1)} {\Gamma(-\alpha+1)i!}   \frac{\theta^{i-A}}{A-i} .
\]
In the notation of Theorem \ref{SD}, we have $\beta_v=\tilde{c}_v$, so that the sum over $v$ constitutes the main term in \eqref{asymp2}.

For the second term in \eqref{f-sharp-hyp}, we have
\[
\sum_{m\le x^\theta} \tau_\alpha(m)  \sum_{x^\theta<n\le x/m}  g(n) 
	= \sum_{0\le j<A} c_j \sum_{m\le x^\theta} \tau_\alpha(m)    \frac{x/m}{(\log(x/m) )^{A+j+1}}   
		+ O\left(\frac{x(\log\log x)^{2A+1}}{(\log x)^{2A+1-|\alpha|}}\right)   .
\]
Lemma \ref{tau_alpha(m)/m} implies that 
\als{
\sum_{m\le x^\theta} \frac{\tau_\alpha(m)}{m(\log(x/m) )^{A+j+1}}   
	&= \sum_{\ell=0}^\infty \binom{A+j+\ell}{\ell} \frac{1}{(\log x)^{A+j+\ell+1}}
		\sum_{m\le x^\theta} \frac{\tau_\alpha(m)  (\log m)^\ell}{m} \\
	&= (\log x)^{\alpha-A-j-1}
		\sum_{\ell=0}^\infty \binom{A+j+\ell}{\ell} \frac{\theta^{\alpha+\ell}}{(\alpha+\ell)\Gamma(\alpha)}  \\
	&\qquad	+ o((\log x)^{\Re(\alpha)-A-j-1}) .
}
Since $A>|\alpha|-\Re(\alpha)$, we conclude that
\als{
\sum_{m\le x^\theta} \tau_\alpha(m)  \sum_{x^\theta<n\le x/m}  g(n) 
	&=  c_0x(\log x)^{\alpha-A-j-1}
		\sum_{\ell=0}^\infty \binom{A+\ell}{\ell} \frac{\theta^{\alpha+\ell}}{(\alpha+\ell)\Gamma(\alpha)} \\
	&\qquad	+ o(x(\log x)^{\Re(\alpha)-A-j-1}) .
}
Therefore
\[
 \sum_{n\le x} f(n)  
 	= x \sum_{j=0}^J \beta_j \frac{ (\log x)^{\alpha-j-1} } {\Gamma(\alpha-j)} 
		+ \left(\gamma +o_{x\to\infty}(1)\right) x (\log x)^{\alpha-1-A}  , 
\]
where
\[
\gamma 
	= -e^{i\theta} (\log2)^A G(1) \sum_{\ell=0}^\infty \binom{A+\ell}{\ell} \frac{\theta^{\alpha+\ell}}{(\alpha+\ell)\Gamma(\alpha)}
	-\frac{1}{\Gamma(\alpha)} \sum_{i=0}^J    \frac{\Gamma(-\alpha+i+1)} {\Gamma(-\alpha+1)i!}   \frac{2^{A-i}}{A-i},
\]
since $\kappa_0=1/\Gamma(\alpha)$ and $c_0=-e^{i\theta}(\log2)^A\sum_{m=1}^\infty g(m)/m = -e^{i\theta}(\log 2)^A G(1)$. We fix $\theta\in[1/3,2/3]$ such that the sum over $\ell$ is non-zero. We note that the constant $\gamma$ is a linear function in $G(1)$, which in turn is a continuous function in the parameter $\eps\in[0,1]$. Choosing an appropriate value of $\eps$, we may ensure that $\gamma\neq0$.
%
%
%
%
This concludes the proof.
\end{proof}

\section{Relaxing the conditions on $|f|$}\label{divisor-bound}
 
We conclude this article by showing that Theorem \ref{SD} remains true if we relax the condition $|f|\le\tau_k$ to strictly weaker conditions, which express that $|f|\le\tau_k$ holds in some average sense. A straightforward hypothesis of this kind is
\eq{A1}{
\sum_{\substack{p\le x \\ \nu\ge1}} \frac{|f(p^\nu)|}{p^\nu} \le k\log\log x+O(1)
	\quad \text{ for all }  x\geq 2.
}
We also need to ensure that the $|f(p)|$, and  the $|f(p^\nu)|,\ \nu\geq 2$, do not vary too wildly on average, which follows from the conditions
\eq{A2}{
 \sum_{p\le x} \frac{|f(p)|\log p}{p} \ll \log x
\quad\text{and}\quad
\sum_{p\le x,\ \nu\ge1} \frac{|f(p^\nu)|^2}{p^\nu} = o_{x\to\infty}(\log x) .
}

As in the beginning of Section \ref{aux}, we write $f=\tau_f*R_f$. Then $R_f$ is supported on square-full integers, and we want to be able to say that 
\eq{A2'}{
\sum_{n\le x}|R_f(n)| \ll x^{1-\delta}\quad(x\ge1)
} 
for some fixed $\delta>0$.  We deduce this  from the second hypothesis  in \eqref{A2}:
\als{
\sum_{n\le x}|R_f(n)| \le \sum_{\substack{ab\le x \\ ab\ \text{square-full}}} |f(a)\tau_{-f}(b)|
	&\le \left( \sum_{\substack{n\le x \\ n\ \text{square-full}}} \tau(n)\right)^{1/2}
	 	\left(\sum_{ab\le x}|f(a)|^2|\tau_{-f}(b)|^2\right)^{1/2}  \\
	&\ll x^{1/4+o(1)} \left(\sum_{ab\le x}|f(a)|^2|\tau_{-f}(b)|^2\cdot \frac{x}{ab}\right)^{1/2}  \\
	&\ll x^{3/4+o(1)} 
}
as $x\to\infty$.

Using the condition \eqref{A2'} and the argument in the beginning of Section \ref{aux}, we reduce the problem to estimating the partial sums of $\tau_f$. Hence, as in Section \ref{aux}, we may assume from now on that $f=\tau_f$, so that $\Lambda_f(p^\nu)  = f(p)\log p$. In particular, we note that \eqref{f-alpha} implies that
\eq{f-alpha'}{
\sum_{n\le x} \Lambda_f(n) = \alpha x + O\left(\frac{x}{(\log x)^A}\right) \quad(x\ge2) ,
}
since
\als{
\sum_{p^\nu\le x,\ \nu\ge2} |\Lambda_f(n)|
	&= \sum_{p^\nu\le x,\ \nu\ge2} |f(p)\log p| \ll \sum_{p\le\sqrt{x}} |f(p)|\log x\\
	&\le \sqrt{x}(\log x)\sum_{p\le \sqrt{x}}\frac{|f(p)|}{p} 
	\ll \sqrt{x}(\log x)(\log\log x) .
}
Furthermore, we have the following estimates on the growth of $\Lambda_f$ and $f$:
\eq{f-growth}{
\sum_{n\le x} \frac{|\Lambda_f(n)|}{n} \ll \log x
\quad\text{and}\quad
\sum_{n\le x} \frac{|f(n)|}{n} \ll (\log x)^k \quad(x\ge2 ) ,
}
which follow immediately from the first hypothesis in \eqref{A2}, and from \eqref{A1}, respectively.

A careful examination of the arguments of Section \ref{proof} reveals that relations \eqref{f-alpha'} and \eqref{f-growth} are the only properties of $f$ that we used when showing Theorem \ref{SD} (after its reduction to the case $f=\tau_f$). Therefore, Theorem \ref{SD} can be extended to all multiplicative functions $f$ satisfying \eqref{f-alpha}, \eqref{A1} and \eqref{A2}.


\bibliographystyle{plain}

 \end{document}